\documentclass[11pt]{article}
\usepackage{amsthm, amsfonts, amssymb, latexsym, amsmath}
\usepackage[a4paper, total={6in, 8in}]{geometry}
\usepackage[dvipsnames,svgnames,table]{xcolor}
\usepackage[colorlinks=true,linkcolor=RoyalBlue,urlcolor=RoyalBlue,citecolor=PineGreen]{hyperref}
\usepackage{tikz,pgf}
\usepackage{tikz-cd}
\usetikzlibrary{calc}
\usepackage[nameinlink]{cleveref}
\usepackage{booktabs}
\usepackage{wrapstuff, wrapfig}

\title{On four-rich points defined by pencils}
\author{Michalis Kokkinos and Audie Warren}

\newcommand{\R}{\mathbb{R}}
\newcommand{\C}{\mathbb{C}}

\newtheorem*{theorem*}{Theorem}

\newcommand{\cL}{\mathcal{L}}

\newcommand{\cC}{\mathcal{C}}

\newtheorem{lemma}{Lemma}

\newtheorem{theorem}{Theorem}
\newtheorem{corollary}{Corollary}

\newtheorem{construction}{Construction}

\begin{document}
 \maketitle

\begin{abstract}
    In this paper we study the number of four-rich points  defined by pencils of certain algebraic objects. Our main result concerns the number of four-rich points defined by four sheaves of planes; under certain non-degeneracy conditions, we prove that four sheaves of $n$ planes in $\mathbb P^3$ determine at most $O(n^{8/3})$ four-rich points. We prove this using the four dimensional Elekes-Szab\'{o} theorem. Using the same method, we prove an upper bound on the number of four-rich points determined by four sets of concentric spheres in $\C^3$. Furthermore, using the same technique with the 3-d Elekes-Szab\'{o} theorem, one can prove upper bounds on four-rich points determined by various configurations of lines/circles in the plane $\C^2$; we give one such example, involving two pencils of lines and two pencils of concentric circles in $\C^2$.
\end{abstract}

\section{Introduction}
In this paper, we study the intersection properties of certain families of objects. Specifically, we give various upper bounds for the number of four-rich points defined by finite subsets of one-dimensional families of algebraic objects. An example of such a result is given in \cite{RocheNewton2018ImprovedBF}, where the authors show that given four sets of concurrent lines in $\mathbb R^2$, where each set has size $n$ and the centers are not all collinear, the number of points which lie on a line from each family (i.e. four-rich points) is of size $O(n^{11/6})$. The concept here is that apart from some degenerate situations (e.g. collinear/coplanar points), such configurations cannot define many four-rich points. Our main result concerns the three dimensional version of this problem; we show that four \textit{sheaves of $n$ planes} in complex projective space $\mathbb P^3$ define few four rich points, under certain non-degeneracy conditions. We recall that a sheaf of planes is a set of planes which all contain a common line, and we call this common line the \textit{axial line} of the sheaf. If the four axial lines of the sheaves are pairwise skew, and the sheaves do not all contain a common transversal line of all axial lines, we prove that we do indeed get few four-rich points. In \Cref{sec:sheafconstruction} we give a construction to show that such restrictions are necessary.

\begin{theorem}\label{thm:sheavesofplanes}
    Let $\Pi_1,\Pi_2,\Pi_3,\Pi_4$ be four sheaves of planes in $\mathbb P^3$, each of size $n$. Suppose that the axial lines of the sheaves are pairwise skew, and that no transversal line of the axial lines is contained in a plane from each sheaf. Then the number of four rich points determined by these sheaves is $O(n^{8/3})$.
\end{theorem}

 Secondly, we consider the number of four-rich points which can be defined by four sets of concentric spheres in $\mathbb C^3$.

\begin{theorem}\label{thm:spheres}
    Let $\mathcal S_1, \mathcal S_2, \mathcal S_3, \mathcal S_4$ be four sets of $n$ concentric spheres in $\mathbb C^3$, whose centres do not lie on a common plane, and no three centres are collinear. Then the number of points in $\mathbb C^3$ which lie in a sphere from each family is $O(n^{8/3})$.
\end{theorem}

As a corollary of \Cref{thm:spheres}, we obtain the following Erd\H{o}s pinned distances type result. In the following, we define $D(q,P)$ for $q\in \mathbb R^3$ and $P \subseteq \mathbb R^3$ to be the set of distances defined between $q$ and $P$.

\begin{corollary}\label{thm:distances}
    Let $P \subseteq \mathbb R^3$ be a finite set of points, and let $q_1,q_2,q_3,q_4 \in \mathbb R^3$ be four non-coplanar points, with no three collinear. Then we have
    $$\max_{i} D(q_i,P) \gg |P|^{3/8}.$$
\end{corollary}

Note that the trivial lower bound is $|P|^{1/3}$, which is tight if three of the points are collinear. This is shown by a simple construction given in \Cref{sec:construction}.

In the last section, we give an upper bound on the number of four-rich points defined by two pencils of lines and two pencils of concentric circles in $\C^2$. Research in this direction was initiated by the seminal paper of Elekes and Szab\'{o} \cite{ElekesSzabo}. Their main result, listed as \Cref{thm:SolymosiZahl} below, was applied to derive an upper bound for the number of intersection points between circles belonging to three families of concentric circles, see \cite[Theorem 32]{ElekesSzabo}. Parallel to \Cref{thm:spheres} and \Cref{thm:distances}, that result can also be seen as a bound on the number of distinct distances between a finite point set in the plane and three other non collinear points. We prove that for two pencils of lines and two pencils of concentric circles of size $n$, the number of four-rich points is $O(n^{12/7})$. We defer the statement of this result to the last section.

\subsection{The Elekes-Szab\'{o} problem}

The Elekes-Szab\'{o} problem \cite{ElekesSzabo} concerns itself with the number of intersection points between an algebraic surface and a finite Cartesian product. The three-dimensional version, concerning the number of intersections between the zero set of a polynomial $F(x,y,z)$ and a finite set of the form $A \times B \times C$, has seen a wide range of applications to combinatorial problems, see for instance \cite{RazSharirSolymosi}, \cite{Roche-Newton-Wong} and the survey paper \cite{deZeeuwsurvey}. We state the most recent version of the Elekes-Szabo theorem due to Solymosi and Zahl \cite{SolymosiZahl}, which builds upon the work of Raz, Sharir, de Zeeuw \cite{RazSharirdeZeeuw}.
\begin{theorem}[Solymosi, Zahl \cite{SolymosiZahl}]\label{thm:SolymosiZahl}
    Let $F\in\C[x,y,z]$ be a degree $d$ irreducible polynomial with each of $F_x,F_y,F_z$ not identically zero. Then one of the two following statements hold.
    \begin{itemize}
        \item For all finite sets $A,B,C\subseteq\C$ with $|A|=|B|=|C|=n$, we have \[|V(F)\cap A\times B\times C|\ll_{d}n^{12/7}.\]
        \item There exists a one-dimensional subvariety $Z_0\subseteq Z(F),$ such that for every $p\in Z(F)\backslash Z_0$, there are open sets $D_1, D_2,D_3$ with $p \in D_1 \times D_2 \times D_3$ and bijective analytic functions $\phi_i:D_i\rightarrow\C,$ with analytic inverses, such that for every $(x,y,z)\in D_1\times D_2\times D_3$ we have\[(x,y,z)\in Z(F)\quad\text{if and only if}\quad\phi_1(x)+\phi_2(y)+\phi_3(z)=0.\]
    \end{itemize}
\end{theorem}

A polynomial $F$ which satisfies item two of \Cref{thm:SolymosiZahl} will be called \textit{degenerate}. In order to apply \Cref{thm:SolymosiZahl} we need a way to test whether polynomials are degenerate. This is provided by the following derivative test. This result can be found in various places in the literature, for instance \cite[Lemma 6]{Roche-Newton-Wong}, \cite[Lemma 33]{ElekesSzabo}, or the survey \cite{deZeeuwsurvey}. We state the variant over the complex numbers which uses rational functions, since this is all we will need in our application.

\begin{lemma}[Degeneracy Test]\label{lem:degentest2d}
    Suppose that $F(x,y,z) \in \mathbb C[x,y,z]$ is a degenerate polynomial. Let $f(x,y)$ be a rational function with $f_x$ and $f_y$ not identically zero, such that wherever $f$ is defined, we have
    $$z = f(x,y) \iff F(x,y,z)=0.$$
    Then there exists a non-empty open set $U \subseteq \C^2$ such that 
    $$\frac{\partial^2 (\ln|f_x/f_y|)}{\partial x \partial y}$$
is identically zero on $U$.    
\end{lemma}

We prove Theorems \ref{thm:sheavesofplanes} and \ref{thm:spheres} by using the four-dimensional Elekes-Szabo result of Raz, Sharir, and de Zeeuw \cite{ES4d}, given as \Cref{thm:4dElekesSzabo} below. 

\begin{theorem}[Raz, Sharir, de Zeeuw \cite{ES4d}] \label{thm:4dElekesSzabo}
    Let $F \in \mathbb C[x,y,z,w]$ be an irreducible polynomial of constant degree, with no partial derivative identically zero. Then one of the following two statements holds.
    \begin{enumerate}
        \item For any finite sets $A,B,C,D \subseteq \mathbb C$, we have
        \begin{multline*}
            |Z(F) \cap A \times B \times C \times D| \ll (|A||B||C||D|)^{2/3}+|A||B|+|A||C|\\+|A||D|+|B||C|+|B||D|+|C||D|.
        \end{multline*}    
        \item There exists a two-dimensional subvariety $Z_0 \subseteq Z(F)$ such that for all $p \in Z(F) \setminus Z_0$, there exist open sets $D_i \subseteq \mathbb C$ and bijective analytic functions $\phi_i: D_i \rightarrow \mathbb C$ with analytic inverses, for $i=1,2,3,4$, such that $p \in D_1 \times D_2 \times D_3 \times D_4$, and for all $(x,y,z,w) \in D_1 \times D_2 \times D_3 \times D_4$ we have
        $$F(x,y,z,w) = 0 \iff \phi_1(x) + \phi_2(y) + \phi_3(z) + \phi_4(w) =0.$$
    \end{enumerate}
\end{theorem}
Informally, the second case of \Cref{thm:4dElekesSzabo} means that (locally) the zero set of $F$ can be represented by a sum of four univariate functions. If this case occurs, we again call the polynomial \textit{degenerate}. In order to use \Cref{thm:4dElekesSzabo}, we begin by proving a Lemma which gives a sufficient condition for a polynomial to be non-degenerate. It is of note that this upcoming degeneracy test appears to be simpler than the analogue \Cref{lem:degentest2d}.

\section{A 4-d Elekes-Szab\'{o} degeneracy test}
In this section we prove the following degeneracy test lemma. The proof is analogous to the three-dimensional case, see for instance \cite{Roche-Newton-Wong}.

\begin{lemma}\label{lem:degentest4d}
   Let $F(x,y,z,w) \in \mathbb C[x,y,z,w]$ be an irreducible polynomial with no partial derivative identically zero, such that we may rearrange the equation $F(x,y,z,w)=0$ to the form $w = f(x,y,z)$ with $f$ rational. If any of the expressions 
   $$\frac{\partial }{\partial z}\left( \frac{f_x}{f_y}\right), \quad \frac{\partial }{\partial x}\left( \frac{f_y}{f_z}\right), \quad \frac{\partial }{\partial y}\left( \frac{f_x}{f_z}\right)$$
   do not vanish on any non-empty open set $U \subseteq \C^3$, then $F(x,y,z,w)$ is non-degenerate.
\end{lemma}

\begin{proof}
    We prove this via contradiction. Suppose that $F(x,y,z,w)$ is a degenerate polynomial, and let $U \subseteq \C^3$ be the Zariski open set where $f(x,y,z)$ is defined. Then for some fixed point $p$ in the set $ \{(x,y,z,w) \in Z(F) \setminus Z_0 : (x,y,z) \in U\}$,  there exist (non-empty analytic) open sets $D_1,D_2,D_3,D_4 \subseteq \mathbb C$ such that $p \in D_1 \times D_2 \times D_3 \times D_4$, chosen small enough so that $D_1 \times D_2 \times D_3 \subseteq U$, and analytic bijective functions with analytic inverses $\phi_i : D_i \rightarrow \mathbb C$, $i=1,2,3,4$, such that on $D_1 \times D_2 \times D_3 \times D_4$, we have
    $$F(x,y,z,w) = 0 \iff \phi_1(x) + \phi_2(y) + \phi_3(z) + \phi_4(w) = 0 \iff w = f(x,y,z).$$
    Defining $\psi(t) = \phi_4^{-1}(-t)$, we see that we have $f(x,y,z) = \psi( \phi_1(x) + \phi_2(y) + \phi_3(z))$ on the open set $D_1 \times D_2 \times D_3$. We then have
    \begin{align*}
            f_x &= \phi_1'(x)\psi'(\phi_1(x) + \phi_2(y) + \phi_3(z)) \\
            f_y &= \phi_2'(y)\psi'(\phi_1(x) + \phi_2(y) + \phi_3(z)) \\
            f_z &= \phi_3'(z)\psi'(\phi_1(x) + \phi_2(y) + \phi_3(z)).
    \end{align*}
    Therefore we have
    $$\frac{f_x}{f_y} = \frac{\phi_1'(x)}{\phi_2'(y)}, \quad \frac{f_y}{f_z} = \frac{\phi_2'(y)}{\phi_3'(z)}, \quad \frac{f_x}{f_z} = \frac{\phi_1'(x)}{\phi_3'(z)}$$
    from which we see that on the open set $(D_1 \times D_2 \times D_3) \setminus (V(f_y) \cup V(f_z)) \subseteq U$, (which is not empty since $V(f_y)$ and $V(f_z)$ are Zariski open subsets of hypersurfaces in $\C^3$ since $F$ has non-vanishing partial derivatives), we must have 
   \begin{equation}\label{degenterms}
       \frac{\partial }{\partial z} \left( \frac{f_x}{f_y} \right) = \frac{\partial }{\partial x} \left( \frac{f_y}{f_z} \right) = \frac{\partial }{\partial y} \left( \frac{f_x}{f_z} \right)= 0.
   \end{equation}
    The lemma then follows.
\end{proof}

\section{Rich points from sheaves of planes}

\subsection{Construction}\label{sec:sheafconstruction}

Before giving the proof of \Cref{thm:sheavesofplanes}, we give a construction showing that certain sets of sheaves of planes can give infinitely many four-rich points - this is important context for the proof of \Cref{thm:sheavesofplanes}. 

\begin{construction}
    There exist four sheaves of planes $\Pi_1,\Pi_2,\Pi_3,\Pi_4$, each of size $n$, whose axial lines are pairwise skew, but define infinitely many four-rich points.
\end{construction}

Let us name the four axial lines $l_1,l_2,l_3,l_4$. Of these four pairwise skew lines, the first three $l_1,l_2,l_3$ define a regulus - recall that this is the union of all lines which intersect all of $l_1,l_2,l_3$, and is a quadratic surface which we will call $R$. There are three possibilities for how the fourth line interacts with this regulus. Firstly, it may be tangential to $R$. Secondly, it could intersect $R$ twice. Lastly, it may be contained in $R$ (and is therefore in the same ruling family of the regulus as $l_1,l_2,l_3$). In any of these cases, there will exist a transversal line of all $l_1,l_2,l_3,l_4$, coming from the second ruling of $R$. A depiction of this line is shown in \Cref{fig:regulus}. The construction is now simple; let $l_5$ be a transversal line of all $l_1,l_2,l_3,l_4$. If we define $\pi_{l_5,l_i}$ to be the plane defined by the two lines $l_5$ and $l_i$ for $i=1,2,3,4$, and we define $\Pi_1,\Pi_2,\Pi_3,\Pi_4$ to be four sheaves of size $n$ with axial lines $l_1,l_2,l_3,l_4$ such that $\pi_{l_5,l_i} \in \Pi_i$, we see that the line $l_5$ is in the intersection of the four sheaves, giving infinitely many four-rich points.

\begin{figure}[h]
    \centering
    \includegraphics[width=0.5\linewidth]{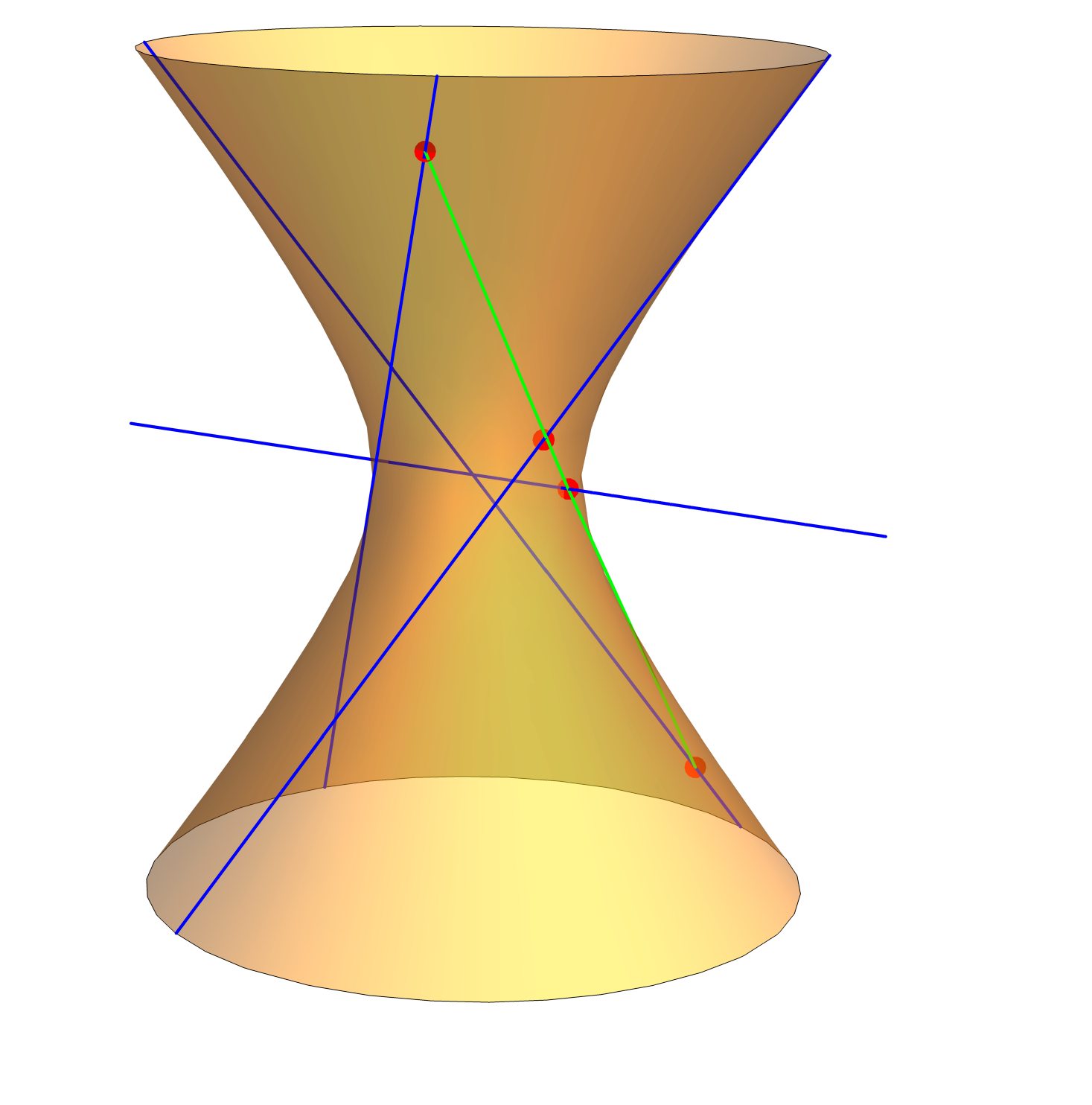}
    \caption{Three skew lines in a regulus, a fourth skew line which intersects the regulus twice, and a green transversal line. The four intersection points are shown in red.}
    \label{fig:regulus}
\end{figure}

\subsection{Proof of \Cref{thm:sheavesofplanes}}

We now give the proof of \Cref{thm:sheavesofplanes}. 

\begin{proof}
    We begin with the four sheaves of planes $\Pi_1,\Pi_2,\Pi_3, \Pi_4$. Let $l_1,l_2,l_3,l_4$ be the corresponding axial lines of these sheaves, which are assumed to be pairwise skew. It is well known that triples of skew lines are projectively equivalent; that is, we may apply a projective transformation $\pi$ such that 
    \begin{align*}
        \pi(l_1) & = Z(x,y) = \left\{[0:0:z:w] \in \mathbb P^3 \right\}\\
         \pi(l_2) & = Z(z,y-w) = \left\{[x:y:0:y] \in \mathbb P^3 \right\}\\
          \pi(l_3) & = Z(x-w,z-w) = \left\{[x:y:x:x] \in \mathbb P^3 \right\}.
    \end{align*}
    We then rename so that $l_1,l_2,l_3,l_4$ refer to the image of the previous lines under $\pi$. Affinely, $l_1$ is given by $x=y=0$, \ $l_2$ is given by $y=1,z=0$, and $l_3$ is given by $x=z=1$. By using the projective transformations which fix the lines $l_1,l_2,l_3$, we can gain further information about the position of $l_4$. Consider the following projective transformations, given in matrix form by 
    $$M_{a,b,c}:= \begin{pmatrix} 1 & b & 0 & 0 \\ a & 1+b-a-c & 0 & 0 \\ 0 & b & 1+b & -b \\ a & b & c & 1-a-c 
        
    \end{pmatrix}$$
    where $a,b,c \in \mathbb C$, and the determinant of the matrix is non-zero. These projective transformations all map the lines $l_1,l_2,l_3$ to themselves. The fourth line $l_4$ exists somewhere in $\mathbb P^3$, and we fix any affine plane (i.e. not the plane at infinity) which contains $l_4$, and does not contain either of the points $[0:0:1:1]$, $[0:0:0:1]$. Note that such a choice always exists since $l_4$ is skew with the $z$ axis $l_1$. We denote this plane by $H$, and write it as $t_1 x + t_2 y + t_3 z + t_4w =0$, for some $t_1,t_2,t_3,t_4 \in \mathbb C$. Since $H$ does not contain the origin we have $t_4 \neq 0$, and so WLOG $t_4 = 1$, and we can write the points on this plane as $$ H = \left\{\left[x:y:z:-t_1 x - t_2 y - t_3 z\right] : [x:y:z] \in \mathbb P^2\right\}.$$
    We now consider the image of this plane under a transformation $M_{a,b,c}$, which are the points in $\mathbb P^3$ of the form
    \begin{equation}\label{eq:imagepoints}
        \begin{pmatrix}
            X \\ Y \\ Z \\ W
        \end{pmatrix} = \begin{pmatrix}
         b y+x \\(b+1) y-y (a+c)+a x \\ b (t_1 x+y (t_2+1)+z (t_3+1))+z \\ b y+(a+c-1) (t_1 x+t_2 y+t_3 z)+a x+c z
    \end{pmatrix}
    \end{equation}
    for $[x:y:z] \in \mathbb P^2.$ Our aim will be to give a transformation $M_{a,b,c}$ such that $M_{a,b,c}(l_4)$ lies in the plane $X+Y+Z =0$ (note that we use capitals for these variables to distinguish them from the points parametrising $H$). This means that we need to find $a,b,c$ such that for all $[x:y:z] \in \mathbb P^2$, the points in \eqref{eq:imagepoints} satisfy $X+Y+Z =0$. We calculate that the sum of the first three coordinates above gives
    $$(1+a+b t_1)x +(1+(3+t_2)b -a-c)y + (1+b(1+t_3))z.$$
    This yields the solution
    \begin{equation}\label{eqn:abcsoln}
        a = \frac{t_1-t_3-1}{t_3+1}, \quad  b = \frac{-1}{t_3 + 1}, \quad c = \frac{-t_1-t_2+2t_3-1}{t_3+1}.
    \end{equation}
    Note that we cannot have $t_3= -1$, since this would imply that $[0:0:1:1]$ lies on $H$, which we have assumed is not the case. Therefore this choice of $a,b,c$ is valid. The determinant of the corresponding matrix $M_{a,b,c}$ is $\frac{(1+t_1+t_2+t_3+t_2t_3)^2}{(1+t_3)^4}$. We will now split into two cases, depending on whether there is some choice of $t_1,t_2,t_3$ such that $\frac{(1+t_1+t_2+t_3+t_2t_3)^2}{(1+t_3)^4} \neq 0$.

    \subsection{Case 1}

    Assume that there exist $t_1,t_2,t_3$ as above, such that $\frac{(1+t_1+t_2+t_3+t_2t_3)^2}{(1+t_3)^4} \neq 0$. In this case, we find a projective transformation $M_{a,b,c}$ with $a,b,c$ as given in \eqref{eqn:abcsoln}, which maps the normalised $l_1,l_2,l_3$ onto themselves, and maps $l_4$ onto the plane $X+Y+Z=0.$ We will now use \Cref{thm:4dElekesSzabo} to bound the number of four-rich points defined by four sheaves of $n$ planes, with the four axial lines $l_1,l_2,l_3,l_4$. Note that projective transformations map planes to planes and lines to lines, and will preserve the number of four rich points. Let us first assume that $l_4$ lies within the affine part of the plane $X + Y + Z = 0$ (that is, $l_4$ does not lie entirely on the plane at infinity). In this case, we can write $l_4$ as the intersection of the plane $X+Y+Z = 0$ and a second, vertical plane $s_1 X + s_2 Y +1 =0$. Note that we can assume the second plane does not pass through the origin.
    
    Consider the first sheaf of planes, $\Pi_1$. These planes all contain the $z$-axis, and any such plane can be described as the zero set of the equation $aX + bY = 0$, for some $[a:b] \in \mathbb P^1$. For our calculations, it will be more convenient to work with planes parametrised by a single complex number, as opposed to $\mathbb P^1$. To do this, we will remove the single plane $Y=0$ from $\Pi_1$ (if it is present in $\Pi_1$), allowing us to write planes in $\Pi_1$ as $X = m_1 Y$ for some complex number $m_1$. By removing any single plane from a sheaf, we are removing at most $O(n^2)$ four-rich points; indeed, for any plane $\pi \in \Pi_1$, the other three sheaves each intersect $\pi$ to give three sets of $n$ lines, call them $L_1,L_2,L_3$ contained in $\pi$. By our assumption on transversal lines, the intersection $L_1 \cap L_2 \cap L_3$ contains no lines, and so has size at most $O(n^2)$. These are precisely the four rich points contained in $\pi$, proving our claim.

    Using the above argument, after removing at most one plane from each sheaf (and therefore removing at most $O(n^2)$ four-rich points, which is far smaller than our claimed bound), we can now write the four sheaves of planes as
    \begin{gather*}
        \Pi_1 = \{ X=m_1 Y : m_1 \in M_1\} \\
        \Pi_2 = \{ Y-1=m_2 Z : m_2 \in M_2\} \\
        \Pi_3 = \{ Z-1=m_3 (X-1) : m_3 \in M_3\} \\
        \Pi_4 = \{ (s_1 + m_4)X + (s_2 + m_4)Y+ m_4 Z + 1=0: m_4 \in M_4\}
    \end{gather*}
    where $M_1,M_2,M_3,M_4 \subseteq \mathbb C$ are finite sets of size $n$. Let $P$ be the set of four rich points. Each point $(x_0,y_0,z_0)\in P$ yields a quadruple $(m_1,m_2,m_3,m_4) \in M_1 \times M_2 \times M_3 \times M_4$ such that the four equations
    \begin{gather}\label{eq:planeequations}
        x_0=m_1 y_0 \\\label{eq:planeequations2}
        y_0-1=m_2 z_0 \\\label{eq:planeequations3}
        z_0-1=m_3 (x_0-1) \\\label{eq:planeequations4}
        (s_1 + m_4)x_0 + (s_2 + m_4)y_0+ m_4 z_0 + 1=0
    \end{gather}
    \sloppy are satisfied. Eliminating $x_0,y_0,z_0$ from this system yields the polynomial equation $F(m_1,m_2,m_3,m_4)=0$, where
    \[\begin{split}
        F(x,y,z,w):=z(xy (w+s_1+1)-xw+y (w+s_2)+w)\\-w (x y+x+y+2)+x (y+1) s_1+y s_2+s_2+1.
    \end{split}\]
    Now consider the projection map
\vspace{-8mm}
    \begin{center}\begin{tikzcd}
        \{ (x_0,y_0,z_0,m_1,m_2,m_3,m_4) \in P \times M_1 \times M_2 \times M_3 \times M_4 :\text{Equations }\eqref{eq:planeequations},\eqref{eq:planeequations2},\eqref{eq:planeequations3},\eqref{eq:planeequations4}\text{ hold} \} \arrow[d, ""] \\
         \{(m_1,m_2,m_3,m_4) \in M_1 \times M_2 \times M_3 \times M_4 : F(m_1,m_2,m_3,m_4) = 0\}.
    \end{tikzcd}
\end{center}
We claim that as long as no transversal line $l$ of $l_1,l_2,l_3,l_4$ is contained in a plane from each of the $\Pi_i$, this projection is a bijection. Indeed, given four planes corresponding to $(m_1,m_2,m_3,m_4)$ in the image set of the projection, these four planes either intersect in a single point, or they intersect in a full line (note that two planes cannot be equal since their axial lines are skew). However we see that if the four planes intersect in a line, then that line is a transversal line to $l_1,l_2,l_3,l_4$. Since we have assumed that this does not happen, the projection above is a bijection. We now aim to upper bound the size of the image set of this projection.

\subsubsection{Irreducibility of $F$}\label{sec:irreducibleF}
In order to prove that $F$ is irreducible, we will use a well know result concerning polynomial resultants. For a proof of the following result, see \cite[Section 3.6, Proposition 1]{coxlittleoshea}. We will freely use this lemma throughout the paper. 

\begin{lemma}
    Suppose that $f,g\in \mathbb C[x,y,z,w]$ are polynomials which have a non-trivial common factor that depends on $x$. Then the resultant $\text{Res}(f,g,x)$ is identically zero.
\end{lemma}

We shall now prove that $F$ is irreducible for all possible values of $s_1,s_2$. Indeed, let us suppose that $F$ is reducible, so that there exist $g,h \in \mathbb C[x,y,z,w]$ non-trivial such that $F = gh$. $F$ itself has a term $-2w$, and so WLOG we may assume that $g$ has a term of the form $\lambda w$ for some $\lambda \neq 0$. Furthermore, $F$ has no term with a factor of $w^2$. This implies that $h$ has no $w$ variable appearing at all, that is, $h \in \C[x,y,z]$. 

Now consider the partial derivative $F_w$. We have $F = gh$, and so $F_w = g_w h + g h_w = g_w h$, since $h_w =0$. Therefore $F$ and $F_w$ have the non-trivial common factor $h\in \C[x,y,z]$. For any of the variables $x,y,z$, we can consider $F$ and $F_w$ as univariate polynomials in that variable. Therefore since $F$ and $F_w$ have the common factor $h$ which depends on at least one of these variables, at least one of the resultants $\text{Res}(F,F_w,x),\text{Res}(F,F_w,y),\text{Res}(F,F_w,z)$ must be identically zero. Calculating these resultants with Mathematica, we find
\begin{gather*}
    \text{Res}(F,F_w,x) = (yz-y-1) (s_1 (y (z-1)+z-2)+s_2 (1+y+z-yz)+y z+z+1) \\
    \text{Res}(F,F_w,y) = (z-xz-1) (1 + s_2 (z-1) + x (1 + s_1 (z-1) + z)) \\
    \text{Res}(F,F_w,z) = (1 - x - x y) (1 + s_2 + y + x (s_1 + y)).
\end{gather*}
For any $s_1,s_2$, all of these resultants are non-zero - this can be easily checked by considering coefficients. For instance, if $\text{Res}(F,F_w,z)$ is identically zero, we must have that its second factor is zero. However this second factor always has a term $xy$, and so can never be zero. Similar arguments hold for the other resultants. Therefore, $F$ is an irreducible polynomial for all values of $s_1,s_2$.

\subsubsection{Applying Elekes-Szab\'{o}}
    
We now need to check whether the irreducible polynomial $F(x,y,z,w)$ is a degenerate polynomial. On the non-trivial Zariski-open (and therefore Euclidean open) subset of $\mathbb C^3$ given by $(x y (w+s_1+1)-x w+y (w+s_2)+w) \neq 0$, we can write $Z(F)$ as the graph of the function $z = f(x,y,w)$ where
    $$f(x,y,w) := \frac{-w (x y+x+y+2)+x (y+1) s_1+y s_2+s_2+1}{x y (w+s_1+1)-x w+y (w+s_2)+w}.$$
    We can now use \Cref{lem:degentest4d}. Using Mathematica, we calculate that the derivative $\frac{\partial}{\partial x}\frac{f_y}{f_w}$ is given by the following expression.
    \begin{figure}[h!]
    \centering
    \includegraphics[width=1\linewidth]{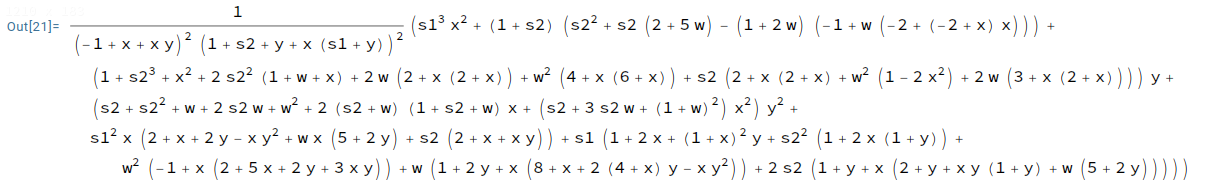}
    \label{fig:enter-label}
    \vspace*{-10mm}
\end{figure}

Using Mathematica, this function is verified to not vanish on any non-trivial open set. Therefore by \Cref{lem:degentest4d} the polynomial $F$ is non-degenerate. Recalling that $P$ denotes the set of four rich points, we then find
$$|P| \leq |Z(F) \cap M_1 \times M_2 \times M_3 \times M_4| \ll n^{8/3}$$
as needed.

\subsubsection{When $l_4$ lies on the plane at infinity}

We now deal with the case where $l_4$ is equal to the line at infinity of the plane $X+Y+Z=0$. If this is the case, the sheaf of planes $\Pi_4$ can be written as (excluding the single plane at infinity, since any single plane can contain only $O(n^2)$ four-rich points under our assumptions)
$$\Pi_4 = \{X+Y+Z=m_4 :\ m_4 \in M_4\}$$
for some finite set $M_4 \subseteq \mathbb C$. Therefore any four rich point $(x_0,y_0,z_0)$ again gives a solution $(m_1,m_2,m_3,m_4) \in M_1\times M_2\times M_3 \times M_4$ to the three equations \eqref{eq:planeequations},\eqref{eq:planeequations2},\eqref{eq:planeequations3}, and the equation
$$x_0+y_0+z_0 = m_4.$$
Following the same procedure as above, we eliminate these four equations down to the single equation
$$F(m_1,m_2,m_3,m_4)=2 + m_2 + m_1 (1 + m_2 + m_3 + m_2 m_3 (m_4 -1)) - (m_3 + m_2 m_3 + m_4)=0.$$
This polynomial $F$ is irreducible, and can be verified as non-degenerate using \Cref{lem:degentest4d}. We do not give the details as the method is identical to that above. Therefore in this case we also find $O(n^{8/3})$ four-rich points.

\subsection{Case 2}

We now consider the case where all planes $t_1 x + t_2 y + t_3 z + 1$ containing $l_4$ satisfy $1+t_1+t_2+t_3+t_2t_3 = 0$, that is, the corresponding matrix $M_{a,b,c}$ with $a,b,c$ as in \eqref{eqn:abcsoln} has determinant zero. We claim that if this is the case, then $l_4$ lies in a common regulus with $l_1,l_2,l_3$. To prove this claim, we use point-plane duality to swap the sheaf of planes containing $l_4$ for a line, call it $L$, in $\mathbb P^3$; specifically the line of points $[t_1:t_2:t_3:1]$ such that the corresponding plane $t_1 x + t_2 y + t_3 z + 1$ contains $l_4$. Note that as we removed the plane containing $l_4$ and the origin, we are missing the one point at infinity of this line. By assumption, this line lies within a doubly ruled surface $R$ given by the equation $1+X + Y + Z + YZ=0$ (in this section we shall use capital letters to distinguish the dual space from the primal space). The equation defining $R$ can be written as $X = -(1+Y)(1+Z)$, and the two ruling families of lines within $R$ can be written as 
\begin{gather*}
\mathcal L_1 := \{l_{\alpha}:\alpha \in \mathbb C\},\quad l_{\alpha} := \{(X,Y,Z) : -(1+Y)=\alpha, \ \alpha(1+Z) = X \}    \\
\mathcal L_2 := \{l_{\beta}':\beta \in \mathbb C\},\quad l_{\beta}' := \{(X,Y,Z) : (1+Z)=\beta, \ -\beta(1+Y) = X \}
\end{gather*}
 We claim that the line $L$ must lie within the family $\mathcal L_1$. Indeed, suppose it lies in $\mathcal L_2$. Then for some $\beta \in \mathbb C$, the sheaf of all planes containing $l_4$ are given by equations of the form
 $$-\beta(1+s)x + sy +(\beta-1)z + 1=0, \quad s \in \mathbb C.$$
    However, all of these planes pass through the point $\left(0,0,\frac{1}{1-\beta}\right)$, implying that $l_4$ must also contain this point, contradicting the skewness of $l_1$ and $l_4$ (recall that after normalisation $l_1$ is the $z$-axis). If we have that $\beta = 1$, the resulting planes all pass through the point $[0:0:1:0]$, yielding the same contradiction. Therefore $L \in \mathcal L_1$.

     Therefore, for some $\alpha \in \mathbb C$ the sheaf of all planes containing $l_4$ are of the form
    \begin{equation} \label{eqn:regulussheaf}
        \alpha(1+s)x - (1+\alpha)y + sz + 1=0, \ \quad s \in \mathbb C.
    \end{equation}
    Excluding the case $\alpha = -1$, the axial line $l_4$ of this sheaf is the set of points 
    $$\left(t,\frac{\alpha t +1}{\alpha + 1},-\alpha t\right), \quad t \in \mathbb C.$$
    The case $\alpha = -1$ corresponds to the axial line $l_4$ being equal to the set of points $(1,y,1)$ for $y \in \mathbb C$, which is the axial line $l_3$ - therefore this case does not happen since $l_3$ and $l_4$ are skew. We now see that all axial lines $l_1,l_2,l_3,l_4$ lie within the regulus given by the equation $xy + xz - yz - x = 0$. Furthermore, the case $\alpha = 0$ can also be excluded, as in this case $l_4$ would be the same as $l_2$.
    We can now continue with the proof as in Case 1. The first three sheaves of planes are the same $\Pi_1,\Pi_2,\Pi_3$ from Case 1, with corresponding sets $M_1,M_2,M_3 \subseteq \mathbb C$. The final sheaf now consists of planes of the form \eqref{eqn:regulussheaf}, with $s \in M_4$ for some finite set $M_4 \subseteq \mathbb C$. As in Case 1, any four-rich point $(x_0,y_0,z_0) \in P$ yields a quadruple $(m_1,m_2,m_3,m_4) \in M_1 \times M_2 \times M_3 \times M_4$ such that the equations
       \begin{gather*}
        x_0=m_1 y_0 \\
        y_0-1=m_2 z_0 \\
        z_0-1=m_3 (x_0-1) \\
        \alpha(1+m_4)x_0 - (1+\alpha)y_0 + m_4 z_0 + 1=0.
    \end{gather*}
    We will again attempt to bound the number of tuples $(x_0,y_0,z_0,m_1,m_2,m_3,m_4)$ by eliminating these four equations down to one equation in $(m_1,m_2,m_3,m_4)$. For this to be valid, we again need the condition that no transversal line of $l_1,l_2,l_3,l_4$ is contained in a plane from each sheaf, so that the analogous projection of that used in Case 1 is also a bijection. If this is the case, we proceed to eliminate down to a single polynomial equation $F(m_1,m_2,m_3,m_4) = 0$, for $F \in \mathbb C[x,y,z,w]$. This polynomial is given by 
    $$F(x,y,z,w) = \alpha (x w+x-1) (y (z-1)-1)+((x-1) z+1) (y-w)$$

\subsubsection{Irreducibility of $F$}
We will use the same method as in \Cref{sec:irreducibleF} to prove that $F$ is irreducible for all $\alpha \in \mathbb C \setminus \{-1,0\}$. As we have seen, these two excluded values of $\alpha$ only occur when $l_4 = l_3$ or $l_4 = l_2$, which is not the case.

Assume that $F$ were reducible, so that $F = gh$ for some non-trivial $g,h \in \mathbb C[x,y,z,w]$. $F$ itself has a term $-w$, and so WLOG we can assume that $g$ has a term of the form $\lambda w$ for some $\lambda \neq 0$. Using the same proof as in \Cref{sec:irreducibleF}, we can see that $h$ must then not contain any occurrence of the variable $w$, so that $h_w = 0$, and therefore we again have that $F$ and $F_w$ have the common factor $h$. Again considering $F$ as univariate in the variables $x,y,z$ in turn, we must have that one of the following resultants is zero.
\begin{gather*}
    \text{Res}(F,F_w,x) = \alpha (1 + \alpha) (1 + y - y z)^2 \\
    \text{Res}(F,F_w,y) = -(1 + \alpha) (1 + (x-1) z)^2 \\
    \text{Res}(F,F_w,z) = -\alpha (xy+x-1)^2.
\end{gather*}
We now see that as long as $\alpha \neq 0,-1$, all of these resultants are non-zero. Therefore $F$ is irreducible for all valid choices of $\alpha$.

\subsubsection{Applying Elekes-Szabo}

We shall now use \Cref{lem:degentest4d} to prove that $F$ is non-degenerate. Indeed, on the Zariski-open set $\mathbb C^3 \setminus Z(\alpha (1 + w) ( y (z-1)-1) + (y-w) z)$, the hypersurface $Z(F)$ can be written as the graph of the function 
$$x = f(y,z,w) = \frac{w - y - \alpha (1 + y) - w z + (1 + \alpha) y z}{\alpha (1 + w) ( y (z-1)-1) + (y-w) z}.$$
We then calculate 
$$\frac{\partial}{\partial y}\left( \frac{f_z}{f_w}\right) = \frac{2 (w-y) (w (z-1)-1)}{\alpha (y (z-1)-1)^3}$$
which does not vanish on any open set (recall that $\alpha \neq 0$). Therefore $F$ is non-degenerate. Applying \Cref{thm:4dElekesSzabo} then shows that the number of four rich points is at most $O(n^{8/3})$, as required. This concludes the proof of \Cref{thm:sheavesofplanes}.

\end{proof}

\section{Rich points from concentric spheres}

In this section we prove \Cref{thm:spheres}, which we restate here for convenience.

\begin{theorem*}
        Let $\mathcal S_1, \mathcal S_2, \mathcal S_3, \mathcal S_4$ be four sets of $n$ concentric spheres in $\mathbb C^3$, whose centres do not lie on a common plane, and no three centres are collinear. Then the number of points in $\mathbb C^3$ which lie in a sphere from each family is $O(n^{8/3})$.
\end{theorem*}

Before proving this result, we give a construction showing that an assumption on the centres of the sets of spheres in necessary.

\subsection{Construction}\label{sec:construction}

 We give the following construction of four points which each define only $|P|^{1/3}$ distances to some point set $P$. Note that the construction is equivalent to giving four pencils of $n$ spheres, with three centres collinear, which define $n^2$ four-rich points.

\begin{construction}
    There exists a set of points $q_1,q_2,q_3,q_4 \in \mathbb R^3$, with $q_1,q_2,q_3$ on a common line, and a set of points $P \subseteq \mathbb R^3$, such that $\max_i (|D(q_i,P)|) \ll |P|^{1/3}$.
\end{construction}

\begin{proof}
    The construction begins in two dimensions; a construction of Elekes \cite{elekescircle} gives three collinear points in the plane $\mathbb R^2 \subseteq \mathbb C^2$, together with $n$ circles around each point, such that there are $\Omega(n^2)$ three-rich points. Indeed, the three points can be taken as $(-1,0),(0,0),$ and $(1,0)$. The circles around each point are defined to have radii $\sqrt{i}$, for $i \in [n]$.

    We now extend this construction to $\R^3$, by giving each point a third coordinate with value $0$, and by defining spheres instead of circles (with the same radii as above). Note that intersecting this construction with any plane containing the $x$-axis yields Elekes' construction in the plane. The three families of $n$ spheres yield $\Omega(n^2)$ three-rich \textit{circles}, which can be seen, as any triple point given by Elekes' construction can be fully rotated around the $x$-axis. We can now choose $q_4$ to be placed on the $y$-axis, very far from the origin. By placing $n$ spheres around $q_4$, each sphere intersects each of the $\Omega(n^2)$ three-rich circles, as long as $q_4$ is far enough away and the radii of the spheres are very close to $||q_4||$. This yields a set of $\Omega(n^3)$ four rich points. Letting $P$ be this set of four-rich points, we see that $P$ only defines $n \sim |P|^{1/3}$ distances to each $q_i$, concluding the proof.
\end{proof}

\subsection{Proof of \Cref{thm:spheres}}

We now prove \Cref{thm:spheres}.

\begin{proof}
    
Let $\mathcal S_1, \mathcal S_2, \mathcal S_3, \mathcal S_4$ be the four sets of concentric spheres, each of size $n$. By translating, we can assume that $\mathcal S_1$ is centred at the origin. We can also rotate and dilate so that $\mathcal S_2$ is centred at $(1,0,0)$. Finally, by rotating around the $x$-axis, we can assume that the spheres from the third set  $\mathcal S_3$ are centred at some point $(a,b,0)$, with $b \neq 0$ since no three centres are collinear. The final set $\mathcal S_4$ is now centred at some point $(c,d,e)$ with $e\neq 0$ since the four points are not coplanar.

Suppose that $(x,y,z)$ is a point lying on a sphere from each family. Let us denote the set of such four-rich points by $P_4$. By identifying a sphere from each family by its squared radius (that is, each set $\mathcal S_i$ has a corresponding set $R_i \subseteq \mathbb C$ of its squared radii), we obtain a solution $(x,y,z,t_1,t_2,t_3,t_4) \in P_4 \times R_1 \times R_2 \times R_3 \times R_4$ to the equations
\begin{align*}
  t_1 &=  x^2 + y^2 + z^2 \\
  t_2 &= (x-1)^2 + y^2 + z^2  \\
   t_3 &= (x-a)^2 + (y-b)^2 + z^2  \\
   t_4&= (x-c)^2 + (y-d)^2 + (z-e)^2.
\end{align*}
Let $A$ be the set of these solutions. Using Mathematica, we can eliminate $x,y,z$ from these equations, to obtain a single polynomial equation $F(t_1,t_2,t_3,t_4)=0$. The polynomial $F(x,y,z,w)$ is shown below. Let $B$ be the set of solutions to $F(t_1,t_2,t_3,t_4) =0$ within $R_1\times R_2 \times R_3 \times R_4$.

\begin{figure}[h]
    \centering
    \includegraphics[width=1\linewidth]{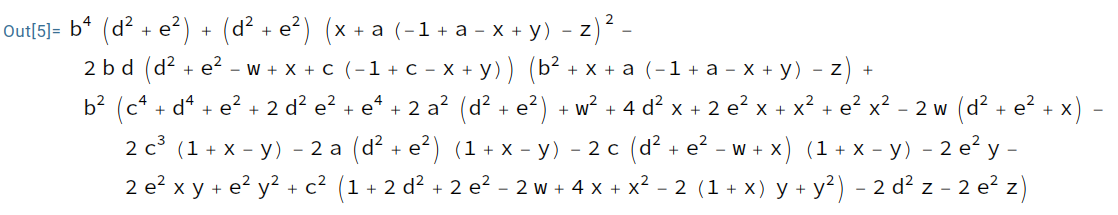}
\end{figure}

As in the proof of \Cref{thm:sheavesofplanes}, we have a corresponding map $\pi:A \rightarrow B$, which is well defined by the reduction above. We claim that any element in the image of $\pi$ has at most four preimages. Indeed, given the quadruple $(t_1,t_2,t_3,t_4)$, the four spheres corresponding to these radii intersect in at most two points; if they were to have a circle in common, the centres would not be in general position.
 
\subsection{Irreducibility of $F$}

We shall now prove that $F$ is irreducible. Assuming that $F$ is in fact reducible, we have that $F = gh$ with $g,h$ non-constant in $\mathbb C[x,y,z,w]$. Since the total degree of $F$ is two, $g$ and $h$ must be linear polynomials. We then must have
$$F = (g_1 x + g_2 y + g_3 z + g_4 w + g_5)(h_1 x + h_2 y + h_3 z + h_4 w + h_5)$$
for some $g_1,...,g_5,h_1,...,h_5 \in \mathbb C$. Using this form as an Ansatz, we can compare coefficients using Mathematica to find that $F$ is reducible if and only if $be = 0$, which is not the case for us. Therefore $F$ is irreducible.

\subsection{Application of Elekes-Szab\'{o}}
We now rewrite $F(t_1,t_2,t_3,t_4)=0$ as $t_4 = f(t_1,t_2,t_3)$ for $(t_1,t_2,t_3)$ on some non-empty Zariski open set. Making use of Mathematica, we can calculate the expression $\frac{\partial}{\partial t_1} \left( \frac{f_{t_2}}{f_{t_3}}\right)$ - it is too long to state here. For $b,e \neq 0$, this expression does not vanish on any open set, and therefore $F(t_1,t_2,t_3,t_4)$ is non-degenerate by \Cref{lem:degentest4d}. By \Cref{thm:4dElekesSzabo}, we then have
$$|P_4| \leq |Z(F) \cap R_1\times R_2 \times R_3 \times R_4| \ll n^{8/3}$$
which concludes the proof.
\end{proof}
\subsection{Pinned distance corollary}
\Cref{thm:distances} can now be easily deduced from  \Cref{thm:spheres}. Indeed, given the finite set $P$, we can cover $P$ with spheres around each $q_1,q_2,q_3,q_4$. One of these covering sets of spheres has the largest size; we add arbitrary spheres to the other sets until they all have this maximum size. Since $P$ is a subset of the four-rich points defined by these spheres, \Cref{thm:spheres} then implies
$$|P| \ll \max_i(|D(q_i,P)|)^{8/3}$$
and the result follows.

Note that it would not be hard to follow the proof of \Cref{thm:spheres} with the sets of spheres being of different sizes, to obtain an `unbalanced' version. We do not pursue this here, in order to simplify the calculations.

\section{Pencils in $\mathbb C^2$}

In this section we prove an upper bound on the number of four rich points determined by two pencils of lines and two pencils of concentric circles within $\C^2$. This follows the work on four-rich points determined by four pencils of $n$ lines in \cite{RocheNewton2018ImprovedBF}. Assuming all the pencils have size $n$, we will use the Elekes-Szab\'{o} Theorem (Theorem \ref{thm:SolymosiZahl}) to obtain an upper bound of $O(n^{12/7})$. We note that using the same proof as below, one can prove many variants, concerning the number of four-rich points defined by different configurations of circles and lines. We prove the following result.

\begin{wrapstuff}[type=figure, height=.13\textheight]
  \includegraphics[scale=0.592]{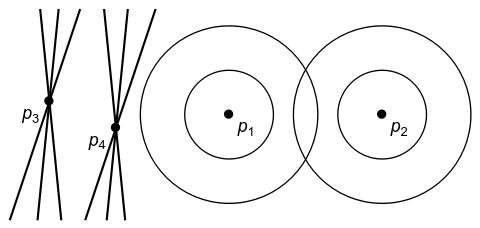}
\end{wrapstuff}

\begin{theorem} Let $\mathcal{C}_1,\mathcal{C}_2$ be two pencils of $n$ concentric circles with distinct centres and non-zero radii, and $\cL_1,\cL_2$ be two pencils of $n$ concurrent lines with distinct centres in the plane $\C^2$. Then the number of points that lie on a curve from each of the four families is $O(n^{12/7})$.\label{thm:2l-2con}
\end{theorem}   
    \begin{proof}
    Identify each circle of $\cC_1,\cC_2$ with its square radius, coming from corresponding sets $R_1,R_2\subseteq\C^*$ respectively, and each line of $\cL_1,\cL_2$ with its slope, coming from the sets $S_1,S_2\subseteq \C$, where all sets have size $n$. We remove any vertical lines, since any single line has at most $2n$ four rich points lying on it, far fewer than we claim. Furthermore, if the line connecting $p_1$ and $p_2$ is in either pencil of lines, we remove it also. Without loss of generality, by rotations, translations and dilations we can assume that the circles in $\cC_1$ are centred at the origin $p_1=(0,0)$ and those in $\cC_2$ have centre $p_2=(1,0)$, and let the centres of $\cL_1,\cL_2$ be at some points $p_3=(a,b),\;p_4=(c,d)$. Let $P_4$ be the set of four rich points. For each four rich point $(x,y) \in P_4$, there exist $(t_1,t_2,t_3,t_4) \in R_1 \times R_2 \times S_1 \times S_2$ such that we have a solution $(x,y,t_1,t_2,t_3,t_4)\in P_4\times R_1\times R_2\times S_1\times S_2$ to the equations
        \begin{gather*}
            x^2+y^2=t_1\\
            (x-1)^2+y^2=t_2\\
            y-b=t_3(x-a)\\
            y-d=t_4(x-c).
        \end{gather*}
Let us call the set of these solutions $A$. Using Mathematica, we can eliminate $x,y,t_2$ from this system of equations and obtain a single expression $F(t_1,t_3,t_4)=0,$ where $F$ is the polynomial\begin{multline*}
    t_3^2 \left(a^2+d^2\right)+t_4^2 \left((t_3 (c-a)+b)^2+c^2-t_1\right)+2 t_3 t_4 (a d t_3+t_1)\\+2 a d t_3+b^2+d^2-2 b (a t_3-c t_4+d t_3 t_4+d)-2 c t_4 \left(a t_3+d t_3^2+d\right)-t_1 t_3^2.
\end{multline*}
Now let $B$ be the set of solutions $(t_1,t_3,t_4) \in R_1 \times S_1 \times S_2$ to $F(t_1,t_3,t_4)=0$, that is,
$$B = Z(F) \cap R_1 \times S_1 \times S_2.$$
 The above reduction shows that we have a map $\pi: A \rightarrow B$ given by $\pi(x,y,t_1,t_2,t_3,t_4)= (t_1,t_3,t_4)$.

\textbf{Claim: } For each $(t_1,t_3,t_4) \in \pi(A) \subseteq B$, there is a unique pre-image under $\pi$.

\vspace{-3mm}

\begin{proof}[Proof of Claim]
Firstly note that if $t_3 = t_4$, then the two lines through $p_3$ and $p_4$ are parallel. They cannot be the same line, since we removed this line if it was present. Therefore the two lines do not intersect, and we find no pre-image.

Now suppose that $t_3 \neq t_4$, so that the two lines through $p_3$ and $p_4$ are not parallel. They therefore have a unique intersection point $(x,y)$, which must be the four-rich point in the preimage. Wherever $p_2$ lies, the entry $t_2$ in the preimage must be the squared distance between $(x,y)$ and $p_2$, which is uniquely determined by $(x,y)$. Therefore we find a unique preimage.
\end{proof}
Given the claim, we have that $|P_4| \leq |A| \leq  |B|$, and so we now aim to upper bound $|B|$ by applying \Cref{thm:SolymosiZahl} to $F$.

We check the irreducibility of $F$ using the method demonstrated in \Cref{sec:irreducibleF}. Assume for contradiction that $F=gh$ for some non-constant polynomials $g,h\in\C[t_1,t_3,t_4].$ Observe that a term involving $t_1$ appears in $F$, but no term involves $t_1^2.$ Hence, without loss of generality, suppose that $g$ contains a non-zero term involving $t_1$, while $h$ does not. The partial derivative of $F$ with respect to $t_1,$ $F_{t_1}=g_{t_1}h$, has the common factor $h$ with $F.$ Consequently, at least one of $Res(F,F_{t_1},t_3)$ and $Res(F,F_{t_1},t_4)$ must be zero. Using Mathematica, we obtain\begin{gather*}
    Res(F,F_{t_1},t_3)=(b - d + (c-a) t_4)^4 (1 + t_4^2)^2\\
    Res(F,F_{t_1},t_4)=(b - d + (c-a) t_3)^4 (1 + t_3^2)^2.
\end{gather*} Both resultants are zero precisely when $a=c$ and $b=d$, which cannot happen due to our condition that the centres of $\cL_1$ and $\cL_2$ are distinct, yielding that $F$ is irreducible.\newline
Finally, solving $F=0$ with respect to $t_1$ yields the equation\begin{multline*}
    t_1=\frac{1}{(t_3-t_4)^2}\Big(a^2 t_3^2t_4^2+a^2 t_3^2-2 a b t_3 t_4^2-2 a b t_3-2 a c t_3^2 t_4^2-2 a c t_3 t_4+2 a dt_3^2 t_4\\+2 a d t_3+b^2 t_4^2+b^2+2 b c t_3 t_4^2+2 b c t_4-2 b d t_3 t_4-2 b d\\+c^2 t_3^2 t_4^2+c^2 t_4^2-2 c d t_3^2 t_4-2 c d t_4+d^2t_3^2+d^2\Big),
\end{multline*}where on the right side we have a rational function to which we apply \Cref{lem:degentest2d}. Under the condition that $(a,b) \neq (c,d)$ the resulting expression does not vanish on any open set, so $F$ is not degenerate. By \Cref{thm:SolymosiZahl} we conclude that \[|P_4| \leq |A| \leq |B| \leq 2|V(F)\cap R_1\times S_1\times S_2|\ll n^{12/7}.\]
    \end{proof}

\section*{Acknowledgments}
The authors were supported by FWF project PAT2559123. We thank Oliver 
Roche-Newton, Jakob F\"{u}hrer, and Paul Hametner for useful discussions.

\vspace{4mm}

\footnotesize{ \noindent \texttt{Michalis Kokkinos
\\Institute for Algebra
\\Johannes Kepler University, Linz, Austria
\\ 	michalis.kokkinos(at)jku.at
\\
\\Audie Warren
\\Johann Radon Institute for Computational and Applied Mathematics (RICAM)
\\Austrian Academy of Sciences, Linz, Austria
\\ audie.warren(at)oeaw.ac.at}}

\normalsize

\bibliography{sheavespaper}
\bibliographystyle{alphaurl}

\end{document}